\documentclass[a4paper,11pt]{article}

\usepackage{latexsym}
\usepackage{amsopn,amscd}
\usepackage{amsfonts}
\usepackage{amssymb}
\usepackage{amsthm}
\usepackage{amsmath}

\def\ETA{\eta}
\def\KK{K}
\def\kk{\mathfrak k}
\def\hh{\mathfrak h}
\def\HH{H}
\def\gg{g}
\def\cc{\zeta}
\def\GG{G}
\def\NN{N}

\def\empha{\em}

\def\bemphas{}
\def\pemphas{}

\def\fra{\mathfrak}

\newtheorem{thm}{Theorem}[section]
\newtheorem{cor}[thm]{Corollary}

\newtheorem{prop}[thm]{Proposition}

\long
\def\MSC#1\EndMSC{\def\arg{#1}\ifx\arg\empty\relax\else
      {\par\narrower\noindent
      2000 Mathematics Subject Classification. #1\par}\fi}

\long
\def\KEY#1\EndKEY{\def\arg{#1}\ifx\arg\empty\relax\else
    {\par\narrower\noindent
      Keywords and Phrases: #1\par}\fi}

\title
{Line bundles on moduli and related spaces}

\author{Johannes Huebschmann \\
Universit\'e des
Sciences et Technologies de Lille\\ UFR de Math\'ematiques, CNRS-UMR 8524
\\ 59655 VILLENEUVE D'ASCQ, C\'edex, France
\\
{Johannes.Huebschmann@math.univ-lille1.fr} }

\usepackage{amssymb}
\usepackage[square,authoryear]{natbib}
\usepackage{amsmath}
\usepackage{amscd}

\numberwithin{equation}{section}

\begin{document}
\setcounter{page}{1}

\maketitle

\begin{abstract}
\noindent
Let $\GG$ be a Lie goup, let
$M$ and $N$ be smooth connected $\GG$-manifolds, let $f \colon M \to
N$ be a smooth $\GG$-map,  and let $P_f$ denote the {\em fiber\/} of
$f$. Given a closed and equivariantly closed relative 2-form for $f$
with integral periods, 
we  construct the principal $\GG$-circle bundles 
with connection on $P_f$ having the given relative
2-form as curvature.
Given a compact Lie group $\KK$,
a biinvariant Riemannian metric on $\KK$,
 and a closed Riemann surface $\Sigma$
of genus $\ell$, when we apply the construction to the particular case where
$f$ is the familiar relator map from $\KK^{2\ell}$ to $\KK$,
which sends the $2\ell$-tuple $(a_1,b_1,\ldots,a_{\ell},b_{\ell})$
of elements $a_j,b_j$ of $\KK$ to $\prod[a_j,b_j]$,
we obtain the principal $\KK$-circle bundles
on the associated extended moduli spaces which, via reduction, then
yield the corresponding line bundles on possibly twisted moduli spaces
of representations of $\pi_1(\Sigma)$ in $\KK$, in particular,
on moduli spaces of semistable holomorphic vector bundles or, more
precisely, on a smooth open stratum when the moduli space is not smooth.
The construction also yields an alternative geometric object,
distinct from the familiar gerbe,
representing the fundamental class in the third integral cohomology group of
$\KK$ or, equivalently, the first Pontrjagin class of the classifying space
of $\KK$.
\end{abstract}

\MSC 

\noindent
Primary: 53D30 

\noindent
Secondary: 14D21 14H60 53D17 53D20 53D50 55N91 55R91 57S25 58D27 81T13
\EndMSC

\KEY Moduli space of central Yang-Mills connections,
moduli space of twisted representations of the fundamental group
of a surface, 
moduli space of semistable holomorphic vector bundles,
symplectic structure of moduli space,
line bundle on moduli space, 
equivariant line bundle,
geometric object representing the first Pontragin class,
symplectic reduction

 \EndKEY

\section{Introduction}

Let $G$ be a Lie group with a bi-invariant Riemannian metric.
Moduli spaces of homomorphisms or more generally twisted
homomorphisms from the fundamental group of a surface to $G$ were
connected with geometry through their identification with moduli
spaces of holomorphic vector bundles \cite{narasesh}. Atiyah and
Bott \cite{atibottw} initiated a new approach to the study of
these moduli spaces by identifying them with moduli spaces of
projectively flat constant central curvature connections on
principal bundles over Riemann surfaces, which they analyzed by
methods of gauge theory. A proof that the resulting symplectic
form is closed, using group cohomology rather than gauge theory,
was given by Karshon \cite{karshone}; in \cite{weinsthi}, A.
Weinstein reformulated Karhon's proof in terms of the double
complex of Bott \cite{bottone} and Shulman \cite{shulmone}. A
purely finite dimensional construction, including that of an
ordinary finite-dimensional Hamiltonian $G$-space, referred to as
an {\em extended moduli space\/}, from which the moduli spaces
arise by ordinary finite-dimensional symplectic reduction, was
announed in \cite{huebjeff} and given in \cite{modus} and
\cite{jeffrtwo}. The construction has been extended in
\cite{guhujewe} to include surfaces with boundary and parabolic
structures along the boundary,
and in \cite{kan} the entire approach has been pushed further 
to handle an arbitrary gauge theory situation
in terms of a similar construction. An application
of the construction in \cite{kan} is a purely combinatorial
construction of the Chern-Simons function over a 3-manifold.

The symplectic structure of the moduli space
(more precisely: on the smooth stratum thereof)
is known to be integral and, at the time,
A. Weinstein raised the issue of constructing a corresponding line
bundle or, equivalently, principal circle bundle. In this paper,
we present a solution to this problem. More precisely, the line
bundle or principal circle bundle not necessarily being defined on
the moduli space itself, we shall construct the requisite
$G$-equivariant circle bundle on the extended moduli space. We
shall actually abstract from the particular case and explore the
more general case of a $G$-equivariant smooth map $f \colon M \to
N$, together with (i) a closed $G$-equivariant relative 2-form 
$(\zeta,\lambda)$ with
integral periods where $\zeta$ is a $\GG$-invariant
2-form on $M$ and $\lambda$
a $\GG$-invariant 3-form on $N$ such that $d \zeta=f^*\lambda$ and with
(ii) the requisite additional technical ingredient
encapsulating the information
to carry out the construction of the principal circle bundle
$\GG$-equivariantly; this additional information
is encoded in a $\GG$-equivariant 
linear map $\vartheta$ from the Lie algebra $\fra \gg$ of $\GG$ to the space
of 1-forms on $N$ and contains the information needed to construct
a $\GG$-momentum mapping from $P_f$ to $\fra \gg^*$,
that momentum mapping being
the additional constituent to arrive at an equivariantly closed 2-form.

Let $I$ denote the unit interval, let $I^2$ be the ordinary unit
square, and let $j_1\colon I \to I^2$ be the injection which sends
the point $t$ of $I$ to $(t,0)\in I^2$. We shall construct the
total space of the circle bundle on the fiber $P_f$ of the map $f$
as a space of equivalence classes of {\em strings\/} of the kind
\begin{equation*}
\CD \{0\} @>>>I @>{j_1}>> I^2
\\
@VVV @VVwV @VV{\phi}V
\\
\{o\} @>>> M @>f>> N.
\endCD
\end{equation*}
In Section \ref{sect1} below,
we recall how a principal circle bundle can be recovered
from the curvature;
thereafter We  refine the construction to an equivariant one.
In Section \ref{fiber} we generalize the construction to that of a
principal circle bundle on the fiber $P_f$ of a map
$f\colon M \to N$ and in Section \ref{equiv} we give the equivariant
extension of that construction.
In Section \ref{liegp} we explore the special case where the target $N$
of $f$ is a Lie group.

Given a closed Riemann surface $\Sigma$ of genus $\ell$ and a compact 
connected Lie group $\KK$,
we take $M= \KK^{2\ell}$, $N= \KK$, and 
$f$ to be the familiar relator map $\KK^{2\ell} \to \KK$
which sends the $2\ell$-tuple $(a_1,b_1,\ldots,a_{\ell},b_{\ell})$
of elements $a_j,b_j$ of $\KK$ to $\prod[a_j,b_j]$;
moreover we choose a biinvariant Riemannian metric on $\KK$ and
take $\lambda$ to be the fundamental 3-form on $\KK$ and $\zeta$ and
$\vartheta$ the corresponding forms explored in 
\cite{guhujewe},
\cite{modus}, 
\cite{huebjeff},
\cite{karshone},
\cite{weinsthi}.
In that particular case the construction yields a $\KK$-equivariant
principal circle bundle on the fiber of the relator map.

Via the holonomy, the fiber of the relator map is actually
based homotopy equivalent to the space
$\mathrm{Map}^o(\Sigma,B\KK)$ of based maps from $\Sigma$ to the classifying
space $B\KK$ of $\KK$.
Each path component of $\mathrm{Map}^o(\Sigma,B\KK)$ corresponds to
a principal $\KK$-bundle on $\Sigma$ and in fact amounts to
the classifying space of the associated group of based gauge
transformations.
Thus the $\KK$-equivariant principal circle bundle
on $P_f$ induces a $\KK$-equivariant principal circle bundle
on $\mathrm{Map}^o(\Sigma,B\KK)$.
This association can be made functorial in terms of geometric presentations
of the surface variable $\Sigma$;
the  geometric object which  thereby results
represents the cohomology class given by the Cartan
3-form and may thus be viewed as an alternative to the familiar
equivariant gerbe representing the first Pontrjagin class of the
classifying space of $\KK$ \cite{brymcboo}. We explain the details
in Section \ref{pont} below.

The extended moduli spaces lie $\KK$-equivariantly in the fiber
of the relator map, and the $\KK$-equivariant principal circle bundles 
on the extended moduli space we are
looking for are then simply obtained by restriction. 
Details are givem in Section \ref{moduli} below.
In a final section we illustrate our method 
in terms of equivariant circle bundles on coadjoint orbits of the loop group.

The approach in the present paper can be extended to a
construction of principal circle bundles in the more general
situation of {\em equivariant plots\/} for an arbitrary gauge
theory situation of the kind developed in \cite{kan}. An extended
moduli space is a special case of such an equivariant plot. We
plan to come back to this situation elsewhere.

\section{Reconstruction of a circle bundle from the curvature}
\label{sect1}

Our main aim is the construction of principal circle bundles on
the fiber of a map. In this section, we will explain the essence
of the construction, but just over a space rather than over the
fiber of a map. This will help understand the subsequent construction
over the fiber of a map.

Let $I$ be the unit interval and let $\NN$ be a topological space
having suitable local properties so that the constructions below
make sense---$\NN$ being a CW-complex will certainly suffice. By a
{\em path\/} in $\NN$ we mean a (continuous) map $u \colon I \to
\NN$ as usual; then $u(0)$ is the {\em starting point\/} and
$u(1)$ the {\em end point\/}. Occasionally we will refer to both
the starting and end point as {\em end points\/}. Let $o$ be a
point of $\NN$, taken henceforth as {\em base point\/}, let
$P_o(\NN)$ be the space of paths in $\NN$ having starting point
$o$, and let $p_o\colon P_o(\NN) \to \NN$ be the obvious
projection which sends a path to its end point,
well known to be a (Hurewicz) fibration 
onto the path component of $o$ having as fiber
$p_o^{-1}(o)$ the space 
$\Omega_o(\NN)$  of closed based loops in $\NN$, based at $o$.
The space
$P_o(\NN)$, topologized as usual by the compact-open topology, is
contractible, the standard contraction being given by the
operation of contracting a path to its starting point.
A familiar construction yields the universal cover $\widetilde \NN$
of $N$: Identify
two paths $w_1$ and $w_2$ having $o$ as starting point and having
the same end point provided these paths are homotopic relative to
the starting and end points. The space of equivalence classes,
suitably topologized, yields the universal cover of $\NN$, the
covering projection being the obvious map which sends a homotopy
class of paths to the common end point. We will now recall how a
variant of this construction yields the principal circle bundles
on $\NN$.

\subsection{The topological construction}

We will use the notation $B^I=\mathrm{Map}(I,B)$.
Let $p \colon E \to B$ be a map, let $p_0\colon B^I \to B$ be the map 
which sends a path $u \colon I \to B$ to its starting point $u(0)$,
let $E \times_B B^I$ be the associated fiber product, and let
$p^I\colon E^I \to E \times_B B^I$ be the obvious map which sends a path
$w\colon I \to E$ to the pair $(w(0), p \circ w)$.
Recall
that $p$ is a Hurewicz fibration if and only if it
admits a {\em lifting function\/} $\lambda\colon E \times_B B^I \to E^I$,
that is to say, a function $\lambda$ that is
required to be a right-inverse for $p^I$, so that
$p^I \circ \lambda$ is the identity of $E \times_B B^I$;
cf. e.~g. \cite{spanier} (Chap.
2, Theorem 8 p.~92). 
Pick a base point $o$ of $B$ and a base point $o$
of $E$ such that $p(o)=o$, where the notation $o$ is slightly abused.
This choice of base points induces an injection 
$j_o\colon P_o(B) \to E \times_B B^I$ in an obvious manner.
Given the lifting function
$\lambda\colon E \times_B B^I \to E^I$ for $p$ so that, in particular,
$p \colon E \to B$ is a fibration,
the composite
\begin{equation}
\begin{CD}
\gamma\colon P_o(B) @>{j_o}>> E \times_B B^I
@>{\lambda}>> E^I @>{p_o}>> E
\end{CD}
\end{equation}
is a map over $B$ and hence a morphism of fibrations.

Let $S^1$ be the circle group and let $\tau
\colon S \to \NN$ be a topological principal circle bundle. 
Choose a lifting function for $\tau$ and 
pick a pre-image $o$ in $S$ of $o$.
The above construction 
yields a map $\gamma_o\colon
\Omega_o(\NN)\to S^1$ which is a homomorphism relative to
composition of loops, and we will refer to $\gamma_o$ as the {\em
topological holonomy\/} of $\tau$ determined by the lifting
function. The topological holonomy $\gamma_o$ represents an integral
class in $\mathrm H^1(\Omega_o(\NN))$. For dimensional reasons,
the {\em transgression\/} from $\mathrm H^1(\Omega_o(\NN))$ to
$\mathrm H^2(\NN)$, i.~e. the inverse of the {\em suspension\/},
is an isomorphism and, under transgression, the class of $\gamma_o$
goes to the topological characteristic class of $\tau$ in $\mathrm
H^2(\NN)$.

A standard construction recovers  the circle bundle $\tau$ from a
topological holonomy of the kind $\gamma_o$: Identify the two paths
$w_1$ and $w_2$ in $\NN$ having $o$ as starting point and having
the same end point provided the composite $w_2^{-1} w_1$, which is
a closed path in $\Omega_o(\NN)$, has value $1 \in S^1$ under
$\gamma_0$. The above map $\gamma$ from $P_o(\NN)$ to
$S$ passes to a homeomorphism from the space of
equivalence classes in $P_o(\NN)$ onto $S$.

We mention in passing that this notion of topological holonomy led
Stasheff to the development of parallel transport in fiber spaces
\cite{staeigte}.

\subsection{The differential-geometric construction}
\label{difgeo}

We now suppose that $\NN$ is a smooth manifold. Then
lifting functions are provided by the operation of horizontal lift
relative to a connection.  A
variant of the construction of the universal
covering, similar to the topological reconstruction of a principal
circle bundle from its topological holonomy reproduced above,
yields the principal $S^1$-bundles on $N$ with connection having
prescribed curvature.

Let $w_1$ and $w_2$ be two
{\em piecewise smooth\/} paths in $\NN$
having $o$ as
starting point and having the same end point.  
We define a
{\em piecewise smooth\/} homotopy
from $w_1$ to $w_2$ {\em relative to the endpoints\/} 
to be a  piecewise smooth
map
\[
h\colon I \times I \longrightarrow \NN
\]
where the term \lq\lq piecewise smooth\rq\rq\ 
is to be interpreted in terms of some
paving of $I \times I$ consisting of polygons,
such that
\begin{itemize}
\item
$h(t,0)=w_1(t)$, $h(t,1)=w_2(t)$, for every $0 \leq t \leq 1$,
\item
$h(0,s)=o$,  $h(1,s)$ is independent of $s$, for every $0 \leq s
\leq 1$.
\end{itemize}

With this preparation out of the way, 
let $c$ be a closed 2-form on $\NN$ with integral periods.
Let $P_o(\NN)$ now denote the space of {\em
piecewise smooth\/} paths in $\NN$ having starting point $o$, let
$\Omega_o(\NN)$ denote the space of {\em piecewise smooth\/}
closed based loops in $\NN$, based at $o$, and let
$\Omega_o(\NN)_0$ be the subspace of piecewise smooth closed loops
which are homotopic to zero relative $o$. Standard smoothing
arguments show that the inclusions from the various piecewise
smooth path spaces  into the corresponding merely continuous path
spaces are homotopy equivalences \cite{warnbook}. 
Identify two piecewise smooth paths $w_1$ and $w_2$
that are homotopic under a piecewise smooth homotopy
$h$ from $w_1$ to $w_2$ relative to the endpoints 
such that $\int_{I\times I} h^*c$ is an integer.
Since $c$ has
integral periods, this condition does not depend on the choice of
homotopy $h$. Let $\overline S$ denote the space of equivalence
classes, let $\overline \tau \colon \overline S \to \widetilde
\NN$ and $\widehat \tau \colon \overline S \to \NN$ be the obvious
projection maps, and let $\Gamma$ be the space of equivalence
classes of closed loops at $o$.

\begin{prop}
Composition of closed loops turns $\Gamma$ into a group. 
\end{prop}

\begin{proof}
The argument consists in copying the construction of the fundamental group. 
\end{proof}

The
assignment to a closed path $u\colon I \to \NN$ with $u(o)=o$ of
its class in $\Gamma$ yields a surjective map $\Omega_o(\NN) \to
\Gamma$, that to such a closed path $u\colon I \to \NN$ with $u(o)=o$
which is, furthermore, null-homotopic relative to $o$ of 
an integral of the kind
$\int_{I\times I} h^*c$ modulo $\mathbb Z$ yields a surjective map
$\Omega_o(\NN)_0 \to S^1$, and the two maps fit together in the
commutative diagram
\begin{equation*}
\CD
\Omega_o(\NN)_0  @>>> \Omega_o(\NN) @>>> \pi_1(\NN)
\\
@VVV @VVV @VV{\mathrm{Id}}V
\\
S^1 @>>> \Gamma @>>> \pi_1(\NN)
\endCD
\end{equation*}
whose bottom row is a central  extension
\begin{equation}
1 \longrightarrow S^1 \longrightarrow \Gamma \longrightarrow
\pi_1(\NN) \longrightarrow 1 \label{centralext1}
\end{equation}
of Lie groups. Moreover, the familiar composition
of paths
\[
P_o(\NN) \times\Omega_o(\NN) \longrightarrow P_o(\NN)
\]
induces a principal $\Gamma$-action on $\overline S$ turning
\begin{equation}
\widehat \tau \colon \overline S \longrightarrow N \label{pb}
\end{equation}
into a principal $\Gamma$-bundle, and the
restriction of the action to $S^1$ turns $\overline \tau$ into a
principal $S^1$-bundle.

Let $\tau \colon S \to N$ be a principal $S^1$-bundle with a
connection 1-form $\omega$ having curvature $c$. The operation of
horizontal lift relative to $\omega$ furnishes a unique map from
$P_o\NN$ to $S$ which passes to  a map from $\overline S$ to $S$
which, in turn,  fits into a morphism $\widehat \tau \to \tau$ of
principal bundles on $N$, i.~e. into a commutative diagram of the
kind
\begin{equation}
\CD \Gamma @>>> \overline S @>{\overline \tau}>> \NN
\\
@VVV @VVV @VV{\mathrm{Id}}V
\\
S^1 @>>> S @>{\tau}>> \NN .
\endCD
\label{CD4}
\end{equation}
Here the left-hand unlabelled vertical homomorphism from $\Gamma$
to $S^1$ is induced from the holonomy $\Omega_o(\NN) \to S^1$ of
$\omega$. This homomorphism splits the exact sequence
\eqref{centralext1}.

We will denote the de Rham algebra of differential forms on $\NN$
by $\mathcal A(\NN)$. We use the (nowadays)
slightly unusual notation $\mathcal A$ to
avoid notational conflict with the notation $\Omega$ for a based
loop space.

Henceforth we will exploit the theory of {\em differentiable\/}
spaces \cite{chenone}, see also \cite{sourithr} where these spaces
are referred to as \lq\lq {\em espaces diffeologiques\/}\rq\rq.
Below, when we refer to a form on a space which is not a smooth
finite-dimensional manifold in an obvious manner, the term \lq\lq
form\rq\rq\ is understood relative to an obvious differentiable
structure. In this vein, view $P_{o}(N)$ as a differentiable space
in the obvious way, and let $\mathcal A^*(P_{o}(N))$ be the
algebra of differential forms on $P_{o}(N)$, relative to the
differentiable structure. Let
\begin{equation}
\ETA\colon \mathcal A^*(P_{o}(N)) \to \mathcal A^{*-1}(P_{o}(N))
\label{homotopy}
\end{equation}
be the {\em homotopy\/} operator given by integration along the
paths which constitute the points of $P_{o}(N)$, so that
\begin{equation}
d\ETA + \ETA d =\mathrm{Id} . \label{homotopy0}
\end{equation}
Thus, integration of $c$ along the paths which constitute the
points of $P_o(\NN)$ yields the 1-form $\vartheta_c =\ETA(p_o^*c)$
on $P_o(\NN)$ such that $p_o^*(c) = d\vartheta_c$, and this 1-form
descends to a $\Gamma$-connection form 
\[
\overline \omega_c \colon \mathrm T \overline S \longrightarrow \mathbb R
\] 
on
$\overline S$ whose curvature coincides with the 2-form $c$ on
$\NN$. Under \eqref{CD4}, the connection form $\overline \omega_c$
descends to $\omega$.

Thus we conclude: {\em The extension\/} \eqref{centralext1} {\em
splits and, given the splitting $\sigma\colon \Gamma \to S^1$, the
induced principal $S^1$-bundle $\tau_{\sigma}=\sigma_*(\widehat
\tau)\colon S_{\sigma} \to \NN$ with connection $\omega_{\sigma}=
\sigma_*(\omega_c)$ has curvature $c$.\/} This recovers the
following classical fact:  

\begin{prop}
The group $\mathrm
H^1(\pi_1(\NN),S^1)=\mathrm{Hom}(\pi_1(\NN),S^1)$ acts simply transitively
on the isomorphism classes of principal $S^1$-bundles with
connection on $\NN$ having curvature $c$. Furthermore, two such
principal $S^1$-bundles with connection are topologically equivalent if and
only if their \lq\lq difference\rq\rq\ in
$\mathrm{Hom}(\pi_1(\NN),S^1)$ lifts to a homomorphism from
$\pi_1(\NN)$ to $\mathbb R$. In particular, when $\NN$ is
simply connected, up to gauge transformation, there is a unique
principal $S^1$-bundle with connection on $\NN$ having curvature
$c$.
\end{prop}

\subsection{The equivariant extension}
\label{equivext}

Let $\GG$ be a Lie group, let $\mathfrak {\gg}$ denote its Lie
algebra, and
suppose that $\NN$ is a (left) $G$-manifold.
Through the associated infinitesimal $\fra g$-action ${\mathfrak{\gg}\to
\mathrm{Vect}(\NN)}$ induced by the $\GG$-action on $\NN$,
the algebra $C^{\infty}(\NN)$ acquires a (right) $\fra g$-module
structure; let
\[
d_{\mathfrak{\gg}} \colon \mathrm{Alt}(\mathfrak
g,C^{\infty}(\NN))\longrightarrow \mathrm{Alt}(\mathfrak
g,C^{\infty}(\NN))
\]
be the resulting (Cartan-Chevalley-Eilenberg) Lie algebra
cohomology operator.

Let $c$ be a $G$-invariant 2-form on $\NN$ and 
let $\tau \colon S \to \NN$ be a principal
$S^1$-bundle on $\NN$ with connection $\nabla$ having curvature $c$. Let
$\GG_{\tau}$ denote the group of pairs $(\phi,x)$ where $\phi
\colon S \to S$ is a bundle automorphism which, on the base $\NN$,
descends to the diffeomorphism $x_{\NN}$ induced from $x \in \GG$.
Since $c$ is $\GG$-invariant, the obvious map from $\GG_{\tau}$ to
$\GG$ is surjective and hence fits into the group extension
\begin{equation}
1 \longrightarrow \mathcal G(\tau) \longrightarrow \GG_{\tau}
\longrightarrow \GG \longrightarrow 1 \label{ext1}
\end{equation}
where $\mathcal G(\tau)\cong \mathrm{Map}(\NN, S^1)$ is the
(abelian) group of gauge transformations of $\tau$. Here
conjugation in $\GG_{\tau}$ induces the obvious action of $G$ on
$\mathcal G(\tau)$ coming from the $G$-action on $\NN$. 
Let $\mathfrak g(\tau)\cong \mathrm{Map}(\NN, \mathbb R)
=C^{\infty}(\NN)$ be the abelian Lie algebra of infinitesimal
gauge transformations of $\tau$, made into a $\GG$- and hence
$\mathfrak{\gg}$-module in the obvious manner. 
The
Lie algebra
extension
associated
with the group extension \eqref{ext1} takes the form
\begin{equation}
0 \longrightarrow \mathfrak {\gg}(\tau) \longrightarrow \mathfrak
{\gg} _{\tau} \longrightarrow \mathfrak {\gg} \longrightarrow 0.
\label{ext2}
\end{equation}
Through the
infinitesimal $\mathfrak{\gg}$-action  on $\NN$, the
connection $\nabla$ induces a section $\nabla_{\mathfrak{\gg}}\colon
\mathfrak{\gg} \to \mathfrak {\gg} _{\tau}$ for \eqref{ext2} in
the category of vector spaces, and we denote by
${c_{\mathfrak{\gg}}\in \mathrm{Alt}^2(\mathfrak
g,C^{\infty}(\NN))}$ the $C^{\infty}(\NN)$-valued Lie algebra
2-cocycle on $\mathfrak{\gg}$ determined by
$\nabla_{\mathfrak{\gg}}$ and the Lie algebra extension
\eqref{ext2}.

For $X\in \mathfrak{\gg}$, we will denote by $X_{\NN}$ the
associated fundamental vector field on $\NN$.  Recall that a {\em
momentum mapping\/} for $c$ (whether or not $c$ is non-degenerate)
is a $\GG$-equivariant map $\mu \colon \NN \to \mathfrak{\gg}^*$
such that the adjoint $\mu^{\sharp}\colon \mathfrak{\gg} \to
C^{\infty}(\NN)$ satisfies the identity
\[
d(\mu^{\sharp}(X)) = c(X_{\NN},\,\cdot\,),\ X \in \mathfrak{\gg};
\]
we will then refer to $\mu^{\sharp}$ as a {\em comomentum\/}. The
connection $\nabla$ and the {2-form} $c$ being fixed, the comomenta
 are precisely the $\GG$-equivariant
$C^{\infty}(\NN)$-valued 1-cochains $\delta$ on $\mathfrak{\gg}$ such that
$d_{\mathfrak{\gg}}(\delta) = c_{\mathfrak{\gg}}\in \mathrm{Alt}(\mathfrak
g,C^{\infty}(\NN))$;
in particular, each such comomentum
\[
\delta\colon \mathfrak g \longrightarrow \mathfrak g(\tau)\cong
\mathrm{Map}(\NN, \mathbb R) =C^{\infty}(\NN) \] yields the Lie
algebra section
\begin{equation}
\nabla_{\mathfrak{\gg}} + \delta \colon \mathfrak{\gg}
\longrightarrow \mathfrak {\gg} _{\tau} \label{section1}
\end{equation} for
the Lie algebra extension \eqref{ext2}. These observations entail
the following well known fact: 

\begin{prop}
When $\GG$ is connected, a
momentum mapping $\mu \colon \NN \to \mathfrak{\gg}^*$ induces a
lift of the $\GG$-action to an action of a suitable covering group
$\widetilde {\GG}$  on $S$ compatible with the $S^1$-bundle
structure and thus turning $\tau$ into a $\widetilde G$-equivariant
principal $S^1$-bundle, and every such lift induces a momentum mapping.
Furthermore, the connection $\nabla$ on $\tau$ is then  as well
$\widetilde {\GG}$-invariant.
\end{prop}

\section{Circle bundles on the fiber of a map}\label{fiber}

Let $M$ and $N$ be smooth connected manifolds, let $f \colon M \to
N$ be a smooth map,  and let $P_f$ denote the {\em fiber\/} of
$f$, made precise below. Given a closed relative 2-form for $f$
(made precise below) with integral periods, using a variant of the
method in the previous section, we will construct the
principal $S^1$-bundles with connection on $P_f$ having the given relative
2-form as curvature.

Recall that, for a space $Y$ and a point $y$ of $Y$, the obvious projection map
$
\pi_y\colon P_y(Y)\longrightarrow Y
$
which sends a path to its end point is a Hurewicz 
fibration onto the path component of $y$ and
the fiber over the point $y$ amounts to the space
$\Omega_y(Y)$ of based loops in $Y$, based at $y$. We will denote
by $i_y\colon\Omega_y(Y)\to P_y(Y)$ the corresponding inclusion.

Let momentarily $M$ and $N$ be merely pathwise connected
topological spaces having suitable local properties, let $f \colon
M \to N$ be a (continuous) map, let $o$ be a point of $M$, taken
henceforth as base point, take $f(o)$ to be the base point of $N$, and
let $P_f$ denote the {\em fiber\/} of the map $f$, that is,
\begin{equation*}
P_f= M \times_N P_{f(o)}(N)=
\{(q,u); u(0) = f(o), u(1) = f(q)\} \subseteq M \times
P_{f(o)}(N) .  
\end{equation*}
We will denote the obvious projection map from $P_f$ to $M$ which
sends the pair $(u,q)\in P_f$ to $q\in M$ by
$
\pi_f\colon P_f \longrightarrow M
$
and we will denote by ${j_f\colon P_f \to
P_{f(o)}(N)}$ the induced map. Since the resulting diagram
\begin{equation}
\CD P_f @>{j_f}>> P_{f(o)}(N)
\\
@V{\pi_f}VV @VV{\pi_{f(o)}}V
\\
M @>>f> N
\endCD
\label{CD0}
\end{equation}
is a pull back square, $\pi_f$ a fibration; the fiber $\pi_f^{-1}(o)$
at the point $o$ of $M$ amounts to the based loop space
$\Omega_{f(o)}(N)$, and we denote by
\[
i_f \colon \Omega_{f(o)}(N) \longrightarrow P_f
\]
the corresponding injection. All these constructions and facts are
well known and classical, cf. e.~g. \cite{spanier}.

Let $j_1\colon I \to I^2$ be the injection which sends $t \in I$
to $(t,0)\in I^2$, and let $E_f$ denote the space of commutative
diagrams
\begin{equation}
\CD \{0\} @>>>I @>{j_1}>> I^2
\\
@VVV @VVwV @VV{\phi}V
\\
\{o\} @>>> M @>f>> N
\endCD
\label{CD1}
\end{equation}
having the property that, for every $0 \leq s \leq 1$, (i)
$\phi(0,s)= f(o)$ and that (ii) $\phi(1,s)$ is independent of $s$.
Thus $E_f$ is a space of {\em strings\/} in $N$ which are subject
to a boundary condition phrased in terms of the map $f$. Given
such a string $(w,\phi)$ of the kind \eqref{CD1}, define
$u_\phi\colon I \to N$ by $u_\phi(t)=\phi(t,1)$ ($t \in I$). Let
$\widehat \pi_f \colon E_f \to P_f$ be the obvious map which sends
$(w,\phi)$ to $(w(1),u_\phi)$.

The space $E_f$ is contractible. Indeed, the assignment to
$(w,\phi,r)\in E_f \times I$ of $(w_r,\phi_r) \in E_F$ given by
$w_r=w$ and $\phi_r(t,s)=\phi(t,(1-r)s)$ contracts $E_f$ onto a
subspace of $E_f$ which amounts to the space $P_oM$ of paths in
$M$ having starting point $o$, and the space $P_oM$, in turn, is
contractible.

Suppose that $P_f$ is pathwise connected. This can always be
arranged for, in the following way: Since $M$ and $N$ are
(supposed to be) pathwise connected, $P_f$ being pathwise
connected is equivalent to the induced map $\pi_1(M) \to \pi_1(N)$
being surjective. When this map is not surjective, let $\overline
N$ be the covering space of $N$ having fundamental group the image
of $\pi_1(M)$ in $\pi_1(N)$; a choice of pre-image $\overline o$
of $f(o)$ then determines a lift $\overline f \colon M \to
\overline N$ of $f$ with $\overline f(o) = \overline o$, the space
$P_{\widehat f}$, viewed as a subspace of $P_f$ in the obvious
way,  is a path component of $P_f$, and every path component of
$P_f$ arises in this manner. Thus we may always assume that $P_f$
is pathwise connected.

With this preparation out of the way, $P_f$ being pathwise
connected, the map $\widehat \pi_f$ is surjective and hence a
fibration. Since the total space $E_f$ is contractible, the
construction in the previous section, with $E_f$ instead of
$P_o(P_f)$, will furnish the principal circle bundles on $P_f$
with connection. We will now explain the details thereof.

Extending the notation $I$ for the unit interval and
$I^2$ for the unit square,  let $I^3$
denote the unit cube and, for $1\leq n \leq 3$, let
\[
J^n=\partial I^{n-1} \times I \cup I^{n-1} \times \{ 1\} \subseteq
I^n ;
\]
here $\partial I^n$ denotes the boundary of $I^n$. 

Let $u_o$ be the constant path in $N$ concentrated at the point
$f(o)$ of $N$, and take $(o,u_o)$ as base point of $P_f$. The
fiber $F_{(o,u_o)}$ over the point $(o,u_o)$ of $P_f$ is the space
of commutative diagrams or {\em strings\/} of the kind
\begin{equation}
\CD J^1 @>>> \partial I^2 @>>> I^2
\\
@VVV @VVwV @VV{\phi}V
\\
\{o\} @>>> M @>f>> N .
\endCD
\label{CD2}
\end{equation}
Here we do not distinguish in notation between the closed path $w
\colon I \to M$ and the map $\partial I^2 \to M$ which coincides
with $w$ on $I\times \{0\}$ and which is constant on $J^1\subseteq
\partial I^2$.
Since $E_f$ is contractible, the fiber $F_{(o,u_o)}$ is homotopy
equivalent to the based loop space $\Omega_{(o,u_o)}P_f$ of $P_f$.
The loop multiplication corresponds to the familiar juxtaposition
of strings of the kind \eqref{CD2}. Consequently the space of
components of the fiber $F_{(o,u_o)}$ underlies the fundamental
group $\pi_1(P_f)$, and  the fundamental group  of each path
component of the fiber $F_{(o,u_o)}$ is given by the second
homotopy group $\pi_2(P_f)$.

The obvious forgetful map from $E_f$ to $P_o(M)$ which sends
$(w,\phi)\in E_f$ to $w\in P_oM$ is a fibration; the fiber
$F_{u_o}$ over the trivial path $u_o$ concentrated at $o\in M$ is
the space of maps from $I^2$ to $N$ having the property that
\[
\phi(0,s) = f(o) =  \phi(1,s) = \phi(t,0),\ 0 \leq t \leq 1, \ 0
\leq s \leq 1.
\]

Given two commutative diagrams $(w_1,\phi_1)$ and $(w_2,\phi_2)$
of the kind \eqref{CD1}, we define a {\em homotopy\/} $(h,H)$ from
$(w_1,\phi_1)$ to  $(w_2,\phi_2)$, written as
\[
(h,H)\colon (w_1,\phi_1) \simeq (w_2,\phi_2),
\]
to be a commutative diagram of the kind
\begin{equation}
\CD \{0\}\times I @>>> I\times I @>{j_1\times\mathrm{Id}}>> I^2
\times I
\\
@VVV @VhVV @VVHV
\\
\{o\} @>>> M @>>f> N
\endCD
\label{CD11}
\end{equation}
subject to the following requirements:
\begin{itemize}
\item
$h(t_1,0) = w_1(t_1)$,\ $0\leq t_1 \leq 1$,
\item
$h(t_1,1) = w_2(t_1)$, \ $0\leq t_1 \leq 1$,
\item
$H(t_1,t_2,0) = \phi_1(t_1,t_2)$, \ $0\leq t_1,t_2 \leq 1$,
\item
$H(t_1,t_2,1) = \phi_2(t_1,t_2)$, \ $0\leq t_1,t_2 \leq 1$,
\item
$h(1,s)$ is independent of $s$, \  $0\leq s \leq 1$,
\item
$H(0,t_2,s)= f(o)$, \ $0\leq t_2,s \leq 1$,
\item
$H(1,t_2,s)$ is independent of $s$, \ $0\leq t_2,s \leq 1$,
\item
$H(t_1,1,s)$ is independent of $s$,\ $0\leq t_1,s \leq 1$.
\end{itemize}

In the same vein, given two commutative diagrams $(w_1,\phi_1)$
and $(w_2,\phi_2)$ of the kind \eqref{CD2}, a {\em homotopy\/}
$(h,H)$ from $(w_1,\phi_1)$ to $(w_2,\phi_2)$, written as
\[
(h,H)\colon (w_1,\phi_1) \simeq (w_2,\phi_2),
\]
is  a commutative diagram of the kind
\begin{equation}
\CD J^1\times I @>>> \partial I^2\times I @>{j\times\mathrm{Id}}>>
I^2 \times I
\\
@VVV @VhVV @VVHV
\\
\{o\} @>>> M @>>f> N .
\endCD
\label{CD12}
\end{equation}
The homotopy classes relative to this homotopy relation are well
known to underly the fundamental group $\pi_1(P_f)$ of $P_f$ or,
equivalently, the second relative homotopy group $\pi_2(f)$ and
hence constitute the space of path components of the fiber $F_{(o,
u_o)}$ of $\widehat \pi_f$. The group structure of $\pi_1(P_f)$ is
induced by the familiar operation of juxtaposition relative to the
first parameter of $I^2$ of two diagrams of the kind \eqref{CD2}
and subsequent reparametrization.

We now suppose that $M$ and $N$ are ordinary smooth manifolds and
that $f$ is a smooth map. With an abuse of notation, we will then
denote by $P_{f(o)}(N)$ the space of {\em piecewise smooth\/}
paths in $N$ having starting point $f(o)$, by $P_f$ the fiber of
$f$ defined merely in terms of {\em piecewise smooth\/} paths in
$N$, and by $E_f$ the space of {\em piecewise smooth\/} strings of
the kind \eqref{CD1}, where the term piecewise smooth is to be
interpreted in terms of a paving of $I^2$.

In the obvious way, we view $P_{f(o)}(N)$ as a differentiable
space in the sense of \cite{chenone}, and we denote by $\mathcal
A^*(P_{f(o)}(N))$ the resulting algebra of differential forms on
$P_{f(o)}(N)$, relative to the differentiable structure. Let
\begin{equation}
\ETA\colon \mathcal A^*(P_{f(o)}(N)) \to \mathcal
A^{*-1}(P_{f(o)}(N)) \label{homotopy1}
\end{equation}
be the {\em homotopy\/} operator given by integration along the
paths which constitute the points of $P_{f(o)}(N)$, so that
\begin{equation}
d\ETA + \ETA d =\mathrm{Id} . \label{homotopy2}
\end{equation}
Let $\lambda$ be a closed 3-form on $N$. Then integration yields
the 2-form
\begin{equation}
\beta_{\lambda}=\ETA (\pi_{f(o)}^*(\lambda)) \label{2form}
\end{equation}
on $P_{f(o)}(N)$ such that $d\beta_{\lambda} =
\pi_{f(o)}^*(\lambda)$, and $\iota_{f(o)}^*\beta_{\lambda}$ is a
closed 2-form on $\Omega_{f(o)}(N)$.

Let $\cc$ be a 2-form on $M$ such that $d\cc = f^*(\lambda) \in
\mathcal A^3(M)$. Then
\begin{equation}
\cc_{(f,\lambda,\cc)}= j_f^*(\beta_{\lambda})- \pi_f^*(\cc)
\label{2form2}
\end{equation}
is a closed 2-form on $P_f$ which restricts to
$j_f^*(\beta_{\lambda})$, that is,
\[
i_f^*(\cc_{(f,\lambda,\cc)}) = j_f^*(\beta_{\lambda}) .
\]
Furthermore, $[\cc_{(f,\lambda,\cc)}]\in
\mathrm H^0(N,\mathrm H^2(P_f))$ transgresses to ${[\lambda]\in \mathrm
H^3(N,\mathrm H^0(P_f))}$.

Suppose that $\lambda$ has integral periods and that, furthermore,
the pair $(\lambda,\cc)$ has integral periods in the sense that,
given a 3-manifold $C$  and a commutative diagram
\begin{equation}
\CD \partial C @>>> C
\\
@VhVV @VVHV
\\
M @>>f> N
\endCD
\label{CD5}
\end{equation}
where $h$ and $H$ are piecewise smooth maps relative to a  paving
of $C$, the difference
\[
 \int_C H^*(\lambda)- \int_{\partial C}h^*(\cc)
\]
is an integer. Then the closed 2-form $\cc_{(f,\lambda,\cc)}$ on
$P_f$ has integral periods. Hence there is a principal circle
bundle on $P_f$ with a connection having curvature
$\cc_{(f,\lambda,\cc)}$. Guided by the considerations in the
previous section, we will now spell out an explicit construction
for all such circle bundles, cf. Corollary \ref{cor} below.

Define the {\em piecewise smooth\/} strings $(w_1,\phi_1)$ and
$(w_2,\phi_2)$ of the kind \eqref{CD1}  to be {\em equivalent\/}
whenever there is a {\em piecewise smooth\/}  homotopy
\[
(h,H)\colon (w_1,\phi_1) \simeq (w_2,\phi_2)
\]
of the kind \eqref{CD11} such that
\[
\int_{I^3} H^*\lambda -\int_{I^2} h^*\cc
\]
is an integer; here $(h,H)$ being piecewise smooth refers to a
suitable paving of the cube $I^3$. Since the pair $(\lambda,\cc)$
has integral periods, this equivalence relation is well defined.
Let $\widehat P_f$ denote the space of equivalence classes of
such strings $(w,\phi)$ of the kind \eqref{CD1}.

Likewise, let $\Gamma_f$ be the space of equivalence classes of
piecewise smooth strings $(w,\phi)$ of the kind \eqref{CD2}, the
equivalence relation being defined in terms of piecewise smooth
homotopies of the kind \eqref{CD12}. The operation of
juxtaposition relative to the first parameter of $I^2$ of two
diagrams of the kind \eqref{CD2} and subsequent reparametrization,
that is, the same operation of composition of diagrams as that
which defines the group structure of $\pi_1(P_f)$, turns $\Gamma_f$
into a group. Moreover, denote by $F_{(o,U_o),0}$ the subspace of
$F_{(o,U_o)}$ of diagrams of the kind \eqref{CD2} that are
null-homotopic; the assignment to such a diagram that is
null-homotopic via a homotopy $(h,H)$ of the difference $ \int_C
H^*(\lambda)- \int_{\partial C}h^*(\cc)$ modulo $\mathbb Z$ yields
a map $F_{(o,U_o),0} \to S^1$, and  this map and the canonical
projection from $F_{(o,U_o)}$ to $\Gamma_f$ fit together in the
commutative diagram
\begin{equation*}
\CD
F_{(o,U_o),0}  @>>> F_{(o,U_o)} @>>> \pi_1(P_f)
\\
@VVV @VVV @VV{\mathrm{Id}}V
\\
S^1 @>>> \Gamma_f @>>> \pi_1(P_f)
\endCD
\end{equation*}
whose bottom row is a central  extension
\begin{equation}
1 \longrightarrow S^1 \longrightarrow \Gamma_f \longrightarrow
\pi_1(P_f) \longrightarrow 1 \label{centralext2}
\end{equation}
of Lie groups, necessarily split.

The  operation
\begin{equation}
E_f \times F_{(o,U_o)} \longrightarrow E_f
\end{equation}
of juxtaposition relative to the first parameter of $I^2$ of a
diagram of the kind \eqref{CD1} and one of the kind \eqref{CD2}
and subsequent reparametrization induces a $\Gamma_f$-action on
$\widehat P_f$. The projection $\widehat \pi_f\colon E_f \to P_f$
descends to a projection $\widehat \tau_f\colon\widehat P_f \to
P_f$.

Given the map $f \colon M \to N$ with $P_f$ not necessarily being
pathwise connected, we can carry out the construction of $\widehat
\tau_f\colon\widehat P_f \to P_f$ separately for each path
component of $P_f$. The fundamental group $\pi_1(P_f)$ does not
depend on the choice of path component.

\begin{thm} \label{thm1} The projection
$\widehat \tau_f\colon\widehat P_f \to P_f$ is a principal
$\Gamma_f$-bundle, and the data determine a $\Gamma_f$-connection,
with connection form $\omega_{(f,\lambda,\cc)}$ on $\widehat P_f$,
having curvature $\cc_{(f,\lambda,\cc)}$.
\end{thm}

\begin{proof} We noted above that (each path component of)
the space $E_f$  is contractible. The construction of the
principal bundle $\widehat \tau_f$ is essentially the same as that
of the principal bundle \eqref{pb} for $\NN=P_f$ in the previous
section, but now carried out with the space $E_f$ rather than with
the space $P_o(P_f)$.
\end{proof}

Now, let $\sigma\colon \Gamma_f \to S^1$ be a splitting of
\eqref{centralext2}. Then the induced principal $S^1$-bundle
$\tau_{\sigma}=\sigma_*(\widehat \tau_f)\colon S_{\sigma,f} \to
P_f$ with connection $\omega_{\sigma,f,\lambda,\cc}=
\sigma_*(\omega_{(f,\lambda,\cc)})$ has curvature
$\cc_{(f,\lambda,\cc)}$. 
Thus we obtain the following.

\begin{cor}\label{cor}
The group
$\mathrm{Hom}(\pi_1(P_f),S^1)=\mathrm{Hom}(\pi_2(f),S^1)$ 
acts simply transitively on the
isomorphism classes of principal $S^1$-bundles with connection on
$P_f$ having curvature $\cc_{(f,\lambda,\cc)}$.
\end{cor}

The construction 
of the principal bundle $\widehat \tau_f$
applies to the particular case where $M$ is a single
point; we then denote by $\Gamma_\Omega$ the resulting group written above
as $\Gamma_f$, and we write the resulting principal
$\Gamma_\Omega$-bundle on $\Omega_o(\NN)$ in the form
\begin{equation}
\widehat \tau_\Omega\colon\widehat {\Omega_o\NN} \longrightarrow \Omega_o(\NN)
\end{equation}
Since $\pi_1(\Omega_o\NN) \cong \pi_2(\NN)$,  the corresponding
central extension \eqref{centralext2} then takes the form
\begin{equation}
1 \longrightarrow S^1 \longrightarrow \Gamma_\Omega
\longrightarrow \pi_2(\NN) \longrightarrow 1. \label{centralext3}
\end{equation}
This extension necessarily splits whence $\Gamma_\Omega$ is an
abelian group.

Return to a general smooth map $f\colon N \to M$.
Let $\pi_M\colon \widetilde M \to M$ and 
$\pi_N\colon \widetilde N \to N$ be the universal covering projections,
pick a base point $\widetilde o$ of $\widetilde M$ over $o$,
lift the map $f$ to a map $\widetilde f \colon \widetilde M \to \widetilde N$,
and take $\widetilde f(\widetilde o)$ to be the base point of $\widetilde N$.
We mention in passing that $\widetilde f$ is determined by 
the value  $\widetilde f(\widetilde o)$. Any string of the kind \eqref{CD1}
admits a unique lift to a string of the kind
\begin{equation}
\CD \{0\} @>>>I @>{j_1}>> I^2
\\
@VVV @VV{\widetilde w}V @VV{\widetilde \phi}V
\\
\{\widetilde o\} @>>> \widetilde M @>{\widetilde f}>> \widetilde N .
\endCD
\label{CD111}
\end{equation}
Consequently
the map $\pi_f$ lifts to a unique based map
\[
\widetilde \pi_f \colon \widehat P_f \longrightarrow \widetilde M,
\]
and the
various fibrations fit into a commutative diagram of the kind
\begin{equation}
\CD
\Gamma_\Omega@>>> \widehat {\Omega_o\NN} @>{\widehat \tau_\Omega}>> \Omega_o(\NN) \\
 @VVV @VVV @VVV
\\
 \Gamma_f @>>> \widehat P_f @>{\widehat \tau_f}>> P_f
 \\
 @VVV @VV{\widehat \pi_f}V  @VV{\pi_f}V
 \\
\pi_1(M) @>>>\widetilde M @>{\pi_M}>> M .
\endCD
\label{CDfibs}
\end{equation}
Here $\widehat \tau_\Omega$, $\widehat \tau_f$, and
$\pi_M$ are principal fiber bundle projections,
$\widehat \pi_f$ and  $\pi_f$ are fibrations, 
and the two left-hand unlabelled  vertical arrows
constitute an exact sequence of groups.

\section{The equivariant extension}
\label{equiv}
As before, let $\GG$ be a Lie group, and denote its
Lie algebra by $\mathfrak{\gg}$. Suppose that $M$ and $N$ are
$\GG$-spaces, that $f$ is a $\GG$-map, and that the forms
$\lambda$ and $\cc$ are $G$-invariant. Then the spaces
$P_{f(o)}(N)$ and $P_f$ are manifestly $G$-spaces, the commutative
diagram \eqref{CD0} is one of $G$-spaces, the de Rham complexes
$\mathcal A^*(M)$, $\mathcal A^*(N)$, $\mathcal A^*(P_{f(o)}(N))$,
$\mathcal A^*(P_f)$, are $\GG$-complexes, and the 2-form
$\cc_{(f,\lambda,\cc)}$ is manifestly $\GG$-invariant. Thus,
to carry out the construction of the principal $\Gamma_f$-bundle 
$\widehat \tau_f$, cf. Theorem \ref{thm1} above, in a $G$-equivariant manner
by means of the method spelled out in subsection \ref{equivext}
above,
we need a momentum mapping
\begin{equation}
\mu \colon P_f \longrightarrow \mathfrak{\gg}^* .
\end{equation}
The resulting $S^1$-bundles of the kind $\tau_{\sigma}$, cf. 
Corollary \ref{cor}, will then
likewise be $\GG$-equivariant.
We will now explain how such a momentum mapping arises.

For $X\in \mathfrak \gg$, let $X_\NN$ be the associated vector
field on $N$ and, given the vector field $Y$ on $N$, let $i_Y$
denote the familiar operator of contraction with the vector field
$Y$. Further, let $\mathcal A^{2j,k}(\NN)$ denote the space of
degree $j$ polynomial maps on $\mathfrak{\gg}$ with values in
$\mathcal A^k(\NN)$, and let
\[
\delta_{\GG}\colon \mathcal A^{2*,*}(\NN)\longrightarrow \mathcal
A^{2*+2,*-1}(\NN)
\]
be the familiar equivariant operator given by
\[
(\delta_{\GG}(\alpha))(X) =-i_{X_{\NN}}(\alpha(X)).
\]
As before, we denote by $\ETA\colon \mathcal A^*(P_{f(o)}(N)) \to
\mathcal A^{*-1}(P_{f(o)}(N))$ the homotopy operator given by
integration along the paths which constitute the points of
$P_{f(o)}(N)$. This homotopy operator is plainly
$\GG$-equivariant.

\begin{thm} \label{thm2} Given a $\GG$-equivariant linear map
 $\vartheta\colon\mathfrak {\gg} \to \mathcal
A^1(\NN)$  satisfying the identities
\begin{equation}
\delta_{\GG}(\lambda) = -d\vartheta,\ \delta_{\GG}(c) =
f^*(\vartheta),\ \delta_{\GG}(\vartheta)=0,
\end{equation}
$\vartheta$ being viewed as a $\mathfrak{\gg}^*$-valued 1-form on
$\NN$, the map
\begin{equation}\mu_{f,\vartheta}=
-\ETA(p^*_{f(o)}(\vartheta))\circ j_f\colon P_f\to
\mathfrak{\gg}^*
\end{equation}
is a $\GG$-momentum mapping for $\cc_{(f,\lambda,\cc)}$.
\end{thm}

\begin{proof}
The conditions on $\vartheta$ say that, somewhat more explicitly,
given $X\in \mathfrak{\gg}$ and vector fields $Y$ and $Z$ on
$\NN$,
\[
\lambda(X_{\NN},Y,Z) = -d(\vartheta(X))(Y,Z), \ c(X_{\NN},Y)
=f^*(\vartheta(X))(Y) .
\]
Then
\begin{equation*}
\begin{aligned} \delta_{\GG}(\cc_{(f,\lambda,\cc)}) &=
j_f^*\delta_{\GG}(\beta_{\lambda}) -\pi_f^*(\delta_{\GG}(c))\\
&=j_f^*\delta_{\GG}(\beta_{\lambda}) -\pi_f^*(f^*\vartheta)
\\
&=j_f^*\delta_{\GG}(\beta_{\lambda})
-j_f^*(\pi^*_{f_{(o)}}\vartheta)
\\
&=j_f^*(\delta_{\GG}(\beta_{\lambda}) -\pi^*_{f_{(o)}}\vartheta)
\end{aligned}
\end{equation*}
View $\vartheta$ as a $\mathfrak{\gg}^*$-valued 1-form on $\NN$
and let $\mu_{\vartheta}=\ETA(-\pi^*_{f(o)}(\vartheta))\colon
P_o(\NN) \to \mathfrak{\gg}^*$ be the $\mathfrak{\gg}^*$-valued
function on $P_o(\NN)$, necessarily $\GG$-equivariant, which
arises by integration of $-\pi^*_{f(o)}(\vartheta) \in \mathcal
A^1(P_o(\NN),\mathfrak{\gg}^*)$. Then
\begin{align*}
\pi^*_{f(o)}(\vartheta)&=d\ETA(\pi^*_{f(o)}(\vartheta))
+\ETA d(\pi^*_{f(o)}(\vartheta)) \\
&= -d\mu_{\vartheta} +\ETA(\pi^*_{f(o)}(d\vartheta)) \\
&= -d\mu_{\vartheta} +\ETA(\pi^*_{f(o)}(\delta_{\GG}(\lambda))
\\
&= -d\mu_{\vartheta} +\delta_{\GG}(\ETA(\pi^*_{f(o)}(\lambda)))
\\
&= -d\mu_{\vartheta} +\delta_{\GG}(\beta_{\lambda})
\end{align*}
whence
\[
d\mu_{\vartheta}= \delta_{\GG}(\beta_{\lambda}) -
\pi^*_{f(o)}(\vartheta).
\]
By definition, $\mu_{f,\vartheta}= \mu_{\vartheta}\circ j_f\colon
P_f\to \mathfrak{\gg}^*$. This map is $\GG$-equivariant. Moreover,
by construction,
\[
d\mu_{f,\vartheta} =j_f^*(\delta_{\GG}(\beta_{\lambda})
-\pi^*_{f_{(o)}}\vartheta) = \delta_{\GG}(\cc_{(f,\lambda,\cc)}),
\]
that is, $\mu_{f,\vartheta}$ is a $\GG$-momentum mapping for
$\cc_{(f,\lambda,\cc)}$.
\end{proof}

Thus, when $\GG$ is connected, as explained in the previous
section, the momentum mapping $\mu_{f,\vartheta}$ furnishes a lift
of the $\GG$-action to $\widehat P_f$ turning
\[
 \widehat \tau_f \colon \widehat P_f \to
P_f,\omega_{f,\lambda,c}
\]
into a $\GG$-equivariant principal
$\Gamma_f$-bundle with connection.

Now, let $\sigma\colon \Gamma_f \to S^1$ be a splitting of
\eqref{centralext2}. Then the induced principal $S^1$-bundle
$\tau_{\sigma}=\sigma_*(\widehat \tau_f)\colon S_{\sigma,f} \to
P_f$ with connection $\omega_{\sigma,f,\lambda,\cc}=
\sigma_*(\omega_{(f,\lambda,\cc)})$ is a bundle in the category of
$\GG$-spaces and has curvature $\cc_{(f,\lambda,\cc)}$. In
particular, the group $\mathrm{Hom}(\pi_1(P_f),S^1)$ acts simply
transitively on the isomorphism classes of $\GG$-equivariant
principal $S^1$-bundles with connection on $P_f$ having curvature
$\cc_{(f,\lambda,\cc)}$.

\section{The case where the target is a Lie group}\label{liegp}

Let $\HH$ be a Lie group and let $\hh$ denote its Lie algebra.
View $\HH$ as an $\HH$-group via conjugation. Let $\,\cdot\,$ be
an invariant symmetric bilinear form on $\hh$, for the moment
neither necessarily non-degenerate nor positive. Let
$\omega_{\HH}$ denote the left-invariant $\hh$-valued
Maurer-Cartan form on $\HH$ and let $\overline \omega_{\HH}$ be
the right-invariant $\hh$-valued Maurer-Cartan form on $\HH$. Let
\[
\lambda = \frac 1{12}[\omega_{\HH},\omega_{\HH}]\cdot
\omega_{\HH}.
\]
This is a closed bi-invariant 3-form on $\HH$.

As before, we denote by $P_e(\HH)$ the space of piecewise smooth
paths in $\HH$ starting at the neutral element $e$ of $\HH$. Let
\[
\vartheta^{\flat}=\frac 12 (\omega_{\HH}+ \overline \omega_{\HH})
\in \mathcal A^1(\HH,\hh).
\]
The resulting map
\[
\ETA(\vartheta^{\flat})\colon P_e(\HH) \longrightarrow\hh
\]
coming from integration sends the piecewise smooth path  $u\colon
I \to \HH$ with $u(0) = e$ to
\[
\ETA(\vartheta^{\flat})(u) = \frac 12 \int_0^1 (u^{-1} \dot u +
\dot u u^{-1}) dt \in \hh.
\]
Notice that when $u$ is the 1-parameter subgroup generated by
$Y\in \hh$ the value $\ETA(\vartheta^{\flat})(u)$ is just $Y$,
that is, the composite of $\ETA(\vartheta^{\flat})$ with the
canonical injection of $\hh$ into $P_e(\HH)$ is the identity of
$\hh$.

Let $\vartheta \in \mathcal A^1(\HH,\hh^*)$ be the $\hh^*$-valued
1-form on $\HH$ which arises from combination of
$\vartheta^{\flat}$ with the adjoint $\hh \to \hh^*$ of the given
invariant symmetric bilinear form. Then
\[
\delta_{\HH} \vartheta = 0, \quad\delta_{\HH}\lambda = -
d\vartheta,
\]
cf. \cite{weinsthi} (4.1), (4.3). It is well known that, when
$\HH$ is compact, an invariant inner product on $\hh$ exists that
is even positive definite, and the resulting 3-form $\lambda$,
occasionally referred to in the literature as the {\em Cartan\/}
3-form, is well known to have integral periods. Thus Theorems
\ref{thm1} and \ref{thm2} apply to maps of the kind $f \colon M
\to \HH$, with $f(o)=e$ (which can alwas be arranged for), under
suitable circumstances. In the remaining sections we spell out
a number of examples.

In the present situation where the target of the map $f$ is a Lie
group, more structure is available: The spaces $P_e(\HH)$ and
$\Omega_e(\HH)$ inherit group structures, and the projection
$\pi_e\colon P_e(\HH)\to \HH$ is a homomorphism whence this
projection is, in particular, a principal fiber bundle with
structure group $\Omega_e(\HH)$. Consequently $\pi_f \colon P_f
\to M$ is necessarily a principal fiber bundle with structure
group $\Omega_e(\HH)$.

Suppose that $\HH$ is connected. Since $\pi_2(\HH)$ is zero, the
extension \eqref{centralext3} comes down to an isomorphism $S^1
\to \Gamma_\Omega$, and the diagram \eqref{CDfibs} takes the form
\begin{equation}
\CD
S^1@>>> \widehat {\Omega_e(\HH)} @>>> \Omega_e(\HH) \\
 @VVV @VVV @VVV
\\
 \Gamma_f @>>> \widehat P_f @>>> P_f
 \\
 @VVV @VVV @VVV
 \\
\pi_1(M) @>>>\widetilde M @>>> M .
\endCD
\label{CDfibs2}
\end{equation}
When $\widehat {\Omega_e(\HH)}$ acquires a group structure--this
will always be so when $\HH$ is simply connected--the resulting
projection $\widehat P_f \to \widetilde M$ is likewise a principal
fiber bundle, with structure group $\widehat{\Omega_e(\HH)}$.
The top row of the diagram
\eqref{CDfibs2}
is then the universal central extension of
the based loop group $\Omega_e(\HH)$ of $H$; 
cf. \cite{lomonesh} and the literature there.

\section{Application to moduli spaces}\label{moduli}

Let $\Sigma$ be a closed surface of genus $\ell$, let $\KK$ be a
compact connected Lie group, let $\,\cdot\,$ be a positive
definite invariant symmetric bilinear form on the Lie algebra
$\kk$ of $K$, and let
\[
\mathcal P=\langle x_1,y_1,\ldots, x_\ell, y_\ell; r\rangle, \ r =
\prod[x_j,y_j],
\]
be the familiar presentation of the fundamental group
$\pi_1(\Sigma)$ of $\Sigma$. The relator $r$ induces a relator map
\[
r\colon\KK ^{2\ell} \longrightarrow \KK.
\]
Endow $\KK$ and $\KK ^{2\ell}$ with the $\KK$-action given by
conjugation in $\KK$. Then $r$ is plainly $\KK$-equivariant. The
construction in Section 2 of \cite{modus} yields a $\KK$-invariant
2-form $\cc$ on $M=\KK ^{2\ell}$ (written there as $\omega_c$)
such that
\[
d\cc = r^*\lambda,
\]
cf. \cite{modus} (18). The pair $(\cc,\lambda)$ arises from a
certain form of total degree 4 on the simplicial model for the
classifying space $B\KK$ of $\KK$ which, in turn, has been
constructed by Bott \cite{bottone}, Dupont \cite{duponone} and
Shulman \cite{shulmone}, and this form on $B\KK$ represents the
universal {\em Pontrjagin\/} class and hence has integral periods.
Consequently the pair $(\cc,\lambda)$ has integral periods in our
sense.

The map between the fundamental groups induced by the relator map
is trivial. Hence a choice of central element $z$ of the universal
covering group $\widetilde \KK$ determines a lift
\[
r_z \colon \KK ^{2\ell} \longrightarrow \widetilde \KK.
\]
The fiber $P_{r_z}$ is connected, even simply connected, since
$\pi_2(\KK)$ is zero and, as $z$ ranges over the center of
$\widetilde \KK$ or, equivalently, over the fundamental group of
$\KK$, the spaces $P_{r_z}$ range over the path components of the
fiber $P_r$ of the original relator map $r$. Furthermore, $P_r$ is
a $\KK$-space in an obvious fashion. 
We can thus take $f$ to be any of the maps $r_z$ 
as $z$ ranges over the center of
$\widetilde \KK$ and apply Theorems \ref{thm1} and \ref{thm2}
with this choice of $f$.

Thus, let $\cc_{(r,\lambda,\cc)}$ be the closed $\KK$-invariant 2-form
\eqref{2form}, constructed separately on each path component of
$P_f$ of the kind $P_{r_z}$. 
This form has integral periods and is necessarily $\KK$-invariant.
Theorem \ref{thm1} exploited separately for each path component of
$P_r$ yields the principal $S^1$-bundle $\tau_r\colon S \to P_r$
together with the connection 1-form $\omega_{(r,\lambda,\cc)}$
having curvature $\cc_{(r,\lambda,\cc)}$ and, since each path
component of $P_r$ is simply connected, this $S^1$-bundle with
connection is uniquely determined by the data up to gauge
transformations. Moreover, 
let
\[
\vartheta \in \mathcal A^1(\KK, \kk^*)\cong \mathcal A^{2,1}(\KK)
\]
be the form introduced in the previous section, where
now $\KK$ is substituted for $H$;
it can be shown that
\[
\delta_{\KK}(\cc)=r^*(\vartheta) \in \mathcal A^{2,1}(\KK^{2\ell}).
\]
Theorem \ref{thm2}, exploited separately for each path component
of $P_r$  of the kind $P_{r_z}$,
 yields a momentum mapping
$\mu_{f,\vartheta}\colon P_r \to \kk^*$ and hence a lift of the
$\KK$-action to a $\widetilde \KK$-action on the total space $S$
of $\tau_r$ turning $\tau_r$ into a $\widetilde \KK$-equivariant
principal $S^1$-bundle with connection. The construction is
entirely rigid and natural in terms of the data.

The extended moduli space $\mathcal H(\mathcal P,\KK)$ constructed
in \cite{modus} embeds  $\KK$-equi\-var\-iantly into $P_r$ in a
canonical manner. The composite of the injection with the momentum
mapping $\mu_{f,\vartheta}$ furnishes the momentum mapping
$\mu^{\sharp}$ constructed in \cite{modus}, and the 2-form
$\cc_{(r,\lambda,\cc)}$ restricts to the 2-form denoted in
\cite{modus} by $\omega_{c,\mathcal P}$. Thus the present
construction recovers the extended moduli space.  {\em New
insight\/} is provided by the functorial construction of the  {\em
principal $S^1$-bundle\/} $\tau_r$: {\em This $S^1$-bundle
restricts to a $\widetilde \KK$-equivariant principal $S^1$-bundle
on $\mathcal H(\mathcal P,\KK)$ having Chern class
$[\omega_{c,\mathcal P}]$ and, in particular, the construction
furnishes, in a functorial manner, a connection having
curvature\/} $\omega_{c,\mathcal P}$.

In a neighborhood of the zero locus of the momentum mapping,
written in \cite{modus} as $\mathcal M(\mathcal P,\KK)$, the
2-form $\omega_{c,\mathcal P}$ is symplectic, and symplectic
reduction yields the corresponding moduli spaces of (possibly)
twisted representations of $\pi_1(\Sigma)$ in $\KK$ \cite{modus}
or, equivalently, the corresponding moduli spaces of central
Yang-Mills connections on $\Sigma$. Thus, {\em reduction carries
the principal $S^1$-bundle to an object which serves as a
replacement for the (in general missing) principal $S^1$-bundle on
the moduli space.\/} When such a bundle exists, the corresponding
line bundle is referred to in the literature as a {\em
Poincar\'e\/} bundle.

On an open and dense stratum, the reduced object is an ordinary
principal $S^1$-bundle, though.

\section{A geometric object realizing the first Pontrjagin class}\label{pont}

Let $o$ be a suitably
chosen base point for $\Sigma$, see below. Under the circumstances
of the previous section, we will construct a $\KK$-equivariant
principal $S^1$-bundle on the space $\mathrm{Map}^o(\Sigma,B\KK)$
of based maps from $\Sigma$ to $B\KK$. The group $\KK$ being
supposed connected, each path component of this space corresponds
to a principal $\KK$-bundle on $\Sigma$ and in fact amounts to the
classifying space of the group of based gauge transformations of
that principal $\KK$-bundle \cite{atibottw}. The space
$\mathrm{Map}^o(\Sigma,B\KK)$ is based homotopy equivalent to
the space $P_r$ in the previous section.

A choice of lifting function for the universal $\KK$-bundle induces
a topological holonomy  $\Omega B \KK \to \KK$. 
Alternatively, 
when $B\KK$ is taken to be the realization of the nerve $N\KK$ of $\KK$,
in each simplicial degree $p$, a suitable connection on the 
corresponding principal $\KK$-bundle on $N_p\KK$
induces a $\KK$-valued holonomy map $\Omega N_p\KK \to \KK$, 
and these combine to a holonomy map $\Omega B\KK \to \KK$.

Let $B^2$ be
the unit 2-disk, identify the circle $S^1$ with the boundary of
$B^2$ in the standard fashion, take $1\in S^1$ as base point of
$S^1$ and $B^2$, let $\vee_{2\ell} S^1$ be the ordinary
one-point union of $2\ell$ copies of the circle $S^1$, let
$F\colon S^1 \to \vee_{2\ell} S^1$ be a map whose cofiber
furnishes the surface $\Sigma$ and,
 for convenience, as base point $o$ for
$\Sigma$ we take the obvious base point of $\vee_{2\ell} S^1$,
viewed as a subspace of $\Sigma$. Thus these spaces fit into the
pushout diagram
\begin{equation}
\begin{CD} S^1 @>>> B^2
\\
@VFVV @VVV
\\
\vee_{2\ell} S^1 @>>> \Sigma
\end{CD}
\label{CD21}
\end{equation}
which we will refer to as a {\em geometric presentation\/} of
$\Sigma$. Take
\[
M= (\Omega B \KK)^{2\ell}=\mathrm{Map}^o(\vee_{2\ell} S^1,B
\KK), \ N= \Omega B \KK=\mathrm{Map}^o(S^1,B \KK),
\]
and let
\[
f=F^*\colon \mathrm{Map}^o(\vee_{2\ell} S^1,B
\KK) \longrightarrow  \mathrm{Map}^o(S^1,B \KK)
\]
be the induced map. Then the fiber $P_f$ of the map $f$ coincides
with the space $\mathrm{Map}^o(\Sigma,B\KK)$. In particular, when
the map $F$ is suitably chosen, the holonomy induces a map from
$\mathrm{Map}^o(\Sigma,B\KK)$ to the space $P_r$ in the previous
section, and this map furnishes the asserted homotopy equivalence.
Abusing notation, let $\lambda\in \mathcal A^3(N)$ be the 3-form
induced from the Cartan 3-form in $\mathcal A^3(\KK)$ via the
holonomy, and let $\cc\in \mathcal A^2(M)$ be the 2-form induced,
via the holonomy, from the 2-form in $\mathcal A^2(\KK^{2\ell})$
denoted by the same symbol. 
Moreover, 
let
\[
\vartheta \in 
\mathcal A^{2,1}(N)
\]
be the form induced, via the holonomy, from the corresponding form in
$\mathcal A^1(\KK, \kk^*)\cong \mathcal A^{2,1}(\KK)$
introduced in Section \ref{liegp} and exploited already
in the previous section
and denoted by the same symbol.
The
construction given in Theorem \ref{thm2},
applied 
separately to each path component of $\mathrm{Map}^o(\Sigma,B\KK)$,
yields a $\widetilde \KK$-equivariant principal $S^1$-bundle 
$\mathrm{Map}^o(\Sigma,B\KK)$ with
connection  having curvature $\zeta_{(f,\zeta,\lambda)}$. The construction is
natural in terms of the data.

This approach relies on the integrality of the fundamental degree
four class, the {\em Pontrjagin\/} class,  of the classifying
space $B\KK$. Our procedure assigns a $\widetilde \KK$-equivariant
principal $S^1$-bundle over $\mathrm{Map}^o(\Sigma,B\KK)$ to {\em
every closed surface\/} $\Sigma$; this assignment is functorial in
terms of geometric presentations of the kind \eqref{CD21} and can
thus be viewed as an {\em alternative\/} to the {\em equivariant
gerbe\/} representing the Pontrjagin class; see \cite{brymcboo}
for the latter.

\section{Coadjoint orbits of the loop group}

Let $M$ be a conjugacy class in $\KK$, let $N=\KK$, and let
$f\colon M \to N$ be the inclusion. There is a 2-form $\cc\in
\mathcal A^2(M)$ 
such that the pair
$(\lambda,\cc)$ has integral periods \cite{guhujewe}. The fiber $P_f$ may be
viewed as a coadjoint orbit of the loop group $\mathcal L\KK =
\mathrm{Map}(S^1,\KK)$. All the requisite requirements are met:
Theorem  \ref{thm1} furnishes a principal $S^1$-bundle with
connection on this coadjoint orbit having 
curvature $\zeta_{(f,\zeta,\lambda)}$ and, when the additional ingredient 
$\vartheta\in \mathcal A^{2,1}(\KK)$  introduced in Section \ref{liegp} 
above is added,
Theorem \ref{thm2} applies equally and yields the
structure of a $\widetilde \KK$-equivariant principal $S^1$-bundle
on $P_f$.

\end{document}